\documentclass[a4paper,12pt]{article}
\usepackage{amssymb}
\usepackage{graphicx}
\usepackage{amsmath,amsthm,amsfonts,amssymb,graphicx}

\newtheorem{theorem}{Theorem}
\newtheorem{lemma}[theorem]{Lemma}
\newtheorem{corollary}[theorem]{Corollary}

\newtheorem*{defn}{Definition}

\newtheorem*{question}{Question}
\theoremstyle{remark}

\newcommand{\N}{{\mathbb N}}

\newcommand{\cP}{\mathcal{P}}

\newcommand{\cA}{\mathcal{A}}
\newcommand{\cB}{\mathcal{B}}
\newcommand{\cC}{\mathcal{C}}
\newcommand{\cD}{\mathcal{D}}
\newcommand{\cI}{\mathfrak{I}}
\newcommand{\cE}{\mathcal{E}}
\newcommand{\cF}{\mathcal{F}}
\newcommand{\cT}{\mathcal{T}}
\newcommand{\cS}{\mathcal{S}}
\newcommand{\cU}{\mathcal{U}}
\newcommand{\cV}{\mathcal{V}}
\renewcommand{\ge}{\geqslant}
\renewcommand{\le}{\leqslant}

\title{Probably Intersecting Families are Not Nested}
\author{  
Paul A.~Russell\footnote{Department
of Pure Mathematics and Mathematical Statistics,
Centre for Mathematical Sciences,
Wilberforce Road,
Cambridge CB3 0WB,
England. \tt P.A.Russell@dpmms.cam.ac.uk}
\and Mark Walters\footnote{School of Mathematical Sciences, Queen Mary,
University of London, London E1 4NS, England
\tt m.walters@qmul.ac.uk}}

\begin{document}
\maketitle

\begin{abstract}
It is well known that an intersecting family of subsets of an
$n$-element set can contain at most $2^{n-1}$ sets. It is natural to
wonder how `close' to intersecting a family of size greater than
$2^{n-1}$ can be.  Katona, Katona and Katona introduced the idea of a
`most probably intersecting family.'  Suppose that $\cA$ is a
family and that $0<p<1$. Let $\cA(p)$ be the (random) family formed by
selecting each set in $\cA$ independently with probability $p$. A
family $\cA$ is \emph{most probably intersecting} if it maximises the
probability that $\cA(p)$ is intersecting over all families of size
$|\cA|$.

Katona, Katona and Katona conjectured that there is a nested sequence
consisting of most probably intersecting families of every possible
size. We show that this conjecture is false for every value of~$p$
provided that $n$ is sufficiently large.

\end{abstract}

We start by recalling the definition of an \emph{intersecting family}: we say
that $\cA\subset\cP([n])$ is \emph{intersecting} if for any $A,A'\in\cA$ we
have $A\cap A'\not=\emptyset$. Since no intersecting family can
contain both a set $A$ and its complement $A^c$, it is easy to see that
there is no intersecting family containing more than $2^{n-1}$
sets. We remark that this upper bound is tight and that, in fact, any
intersecting family can be extended to an intersecting family of this
size.

Having observed this bound, it is natural to wonder how `close' to
intersecting a family of size greater than $2^{n-1}$ can be.  Katona,
Katona and Katona~\cite{3kat} introduced the idea of a \emph{most
  probably intersecting family}. Suppose that $\cA$ is a family and
that $0<p<1$. Let $\cA(p)$ be the (random) family formed by selecting
each set in $\cA$ independently with probability $p$. They asked which
family $\cA$ of  given size maximises the probability that $\cA(p)$
is intersecting.

In the same paper, they solve this problem in cases where $|\cA|$ is
only a little greater than $2^{n-1}$. More precisely, they find
extremal families for $|\cA|\le 2^{n-1}+\binom{n-1}{\lfloor
  (n-3)/2\rfloor}$.

They also conjectured that there are extremal families $\cA_i$ with
$|\cA_i|=2^{n-1}+i$ which are nested: that is, $\cA_i\subset \cA_j$
whenever $i<j$. In their paper it is a little unclear for which $p$
they make this conjecture: there seems to be no reason to believe the
optimal families are the same for different $p$. We remark that there
is a simple counter-example if $p\ll 2^{-n}$ (see the discussion
following Theorem~\ref{t:unique} below) and so clearly this was not
what was meant.  In this paper we prove the following theorem which
shows that the conjecture is false for all~$p$.
\begin{theorem}\label{t:not-nested}
  Suppose that $n\ge 21$. Let $2\le s\le 2^{n-1}$. Then
  \begin{itemize}
  \item 
  $[n]^{(\ge 3)}\cup \{A\in [n]^{(2)}:1\in A\}$ is the unique (up to
    reordering the coordinates) family of size $\sum_{k=3}^n
    \binom{n}{k}+n-1$ maximising the number of intersecting
    subfamilies of size $s$.
\item $[n]^{(\ge 3)}\cup \{A\in [n]^{(2)}:n\not\in A\}$ is the unique
  (up to reordering the coordinates) family of size $\sum_{k=3}^n
  \binom{n}{k}+\binom{n-1}{2}$ maximising the number of intersecting
  subfamilies of size $s$.
  \end{itemize}
\end{theorem}
\noindent%
Clearly these families are not nested, even with a reordering of the
coordinates. 

We can think of forming $\cA(p)$ by first choosing a random variable
$s\sim \textrm{Binom}(|\cA|,p)$ and then choosing $s$ sets uniformly
at random from $\cA$. Hence Theorem~\ref{t:not-nested} shows that
these families are the unique most probably intersecting families for
these two specific sizes for any $0<p<1$.

If $p\ll 2^{-n}$ then the most likely scenario is that $\cA_p$ is
empty and the next most likely is that it consists of a single set. In
each case the subfamily $\cA_p$ is trivially intersecting regardless
of our choice of the original family $\cA$ (assuming, of course, that
$\emptyset\not\in\cA$).  The next most likely case is that there are
exactly two sets in $\cA_p$; this is far more likely than there being
more than two sets.  Hence, for very small $p$, proving
Theorem~\ref{t:not-nested} for the case $s=2$ would give a
counterexample to the conjecture. This was essentially done (except
the uniqueness) by Frankl~\cite{MR0427079} and independently by
Ahlswede~\cite{MR572471}.

Our proof of Theorem~\ref{t:not-nested} consists of two main
steps summarised by the following two theorems.

\begin{theorem}\label{t:unique}
  Suppose that $n\ge 4$, $N\in \N$ and $2\le s\le 2^{n-1}$.  If
\[
    \text{$n=2t$ is even and $N> 2^{n-1}+\tfrac12\tbinom{n}{t}-t$}
\]
or
\[
    \text{$n=2t+1$ is odd and $N> 2^{n-1}+\tbinom{n-1}{t-1}-t-1$}
\]
then any family $\cA\subset \cP([n])$ of size $N$ containing the
maximal number of intersecting subfamilies of size $s$ is of the form
$[n]^{(\ge r+1)}\cup \cB$ where $\cB\subset [n]^{(r)}$ and $r$
satisfies $\sum_{k=r+1}^n\binom nk \le N< \sum_{k=r}^n\binom nk$.
\end{theorem}
In~\cite{PARkat} it is shown that there exists \emph{some} family of
this form maximising the number of intersecting subfamilies of size $s$.
Theorem~\ref{t:unique} strengthens this result by showing that
\emph{all} the optimal families are of this form. This result may be
of interest in its own right.

As we remarked above, the extremal families for $s=2$ have been widely
studied. However, it does not appear to have been proved, even in this
case, that every extremal family must have the above form.

As $\cP([n])$ contains many different intersecting families of order
$2^{n-1}$, we trivially require $N>2^{n-1}$ in Theorem~\ref{t:unique}.
In fact, it is easy to see that a larger lower bound on $N$ is
actually required.  Indeed if we take \emph{any} maximal intersecting
family $\cA_0$ and form the family $\cA$ by adding a maximal set $A$ 
not in $\cA_0$ then this new family is extremal
for all $s$, since $A$ and $A^c$ are the only pair of disjoint sets in
$\cA$. In fact, the bound stated in Theorem~\ref{t:unique} is tight:
in Theorem~\ref{t:construct} we construct, for all appropriate values
of $N$, extremal families which are not of the desired form.

The final step is the following theorem which is at the heart of the
proof.
\begin{theorem}\label{t:layer2}
  Suppose that $n\ge 21$. Let $2\le s\le 2^{n-1}$ and $0\le i\le
  \binom{n}{2}$. Suppose that $\cA\subset \cP([n])$ is any family of
  size $\sum_{k=3}^n \binom{n}{k}+i$ of the form $[n]^{(\ge 3)}\cup
  \cB$ with $\cB\subset [n]^{(2)}$, and that, subject to these
  conditions, $\cA$ contains the maximal number of intersecting
  subfamilies of size $s$. Then $\cB$ is a family of size $i$
  contained in $[n]^{(2)}$ that contains the maximal number of
  intersecting pairs.
\end{theorem}

The families $\cB\subset [n]^{(2)}$ of size $i$ maximising the number
of intersecting pairs are well understood: each is either a
quasi-clique or a quasi-star. We define these terms and discuss for
which $i$ each of these cases occurs after the following proof.

Note that we could rephrase the theorem to say that the family $\cA$
that maximises the number of intersecting subfamilies of size $s$
necessarily also maximises the number of intersecting pairs. This is
clearly equivalent as each set in $\cB$ intersects the same number of
sets in $[n]^{(\ge 3)}=\cA\setminus\cB$.

  Given Theorems~\ref{t:unique} and~\ref{t:layer2}, it is easy to prove
  Theorem~\ref{t:not-nested}.
\begin{proof}[Proof of Theorem~\ref{t:not-nested}]
First suppose that $\cA$ is a family of size $\sum_{k=3}^n
\binom{n}{k}+n-1$ maximising the number of intersecting subfamilies
of size $s$.  Theorem~\ref{t:unique} tells us that $\cA=[n]^{(\ge
  3)}\cup \cB$ for some $\cB\subset [n]^{(2)}$. Clearly we must have
$|\cB|=n-1$. Now Theorem~\ref{t:layer2} tells us that $\cB$ contains
the maximal number of intersecting pairs over all families in
$[n]^{(2)}$ of size $n-1$. It is obvious that $\cB=\{B\in
[n]^{(2)}:1\in B\}$ maximises the number of intersecting pairs, as all
pairs intersect, and, since $|\cB|>3$, that it is the unique (up to
reordering coordinates) family that does. Hence in this case $\cA$
must have the required form.

For the second case, suppose that $\cA$ is a family of size
$\sum_{k=3}^n \binom{n}{k}+\binom{n-1}2$ maximising the number of
intersecting subfamilies of size $s$. As above, we see that
$\cA=[n]^{(\ge 3)}\cup \cB$ for some $\cB\subset [n]^{(2)}$ and that
$\cB$ contains the maximal number of intersecting pairs over all
families in $[n]^{(2)}$ of size $\binom{n-1}{2}$.  Again, the extremal
family $\cB$ is unique up to reordering the coordinates: it consists
of all the 2-sets not containing $n$.  This is a little less obvious
but follows from the result for $i=n-1$ above. Indeed, the family
$\cB$ containing the most intersecting pairs minimises the number of
intersecting pairs with one element in $\cB$ and one element in
$\cB^c$. Thus it also maximises the number of intersecting pairs in
$\cB^c$. By the above, $\cB^c$ is $\{A:1\in A\}$ and the result follows
(after a reordering of the coordinates).
\end{proof}

In fact, the extremal families $\cB$ have been precisely
determined. Suppose $i=\binom a2+b$ with $0\le b<a$. The
\emph{quasi-complete graph} of order $n$ with $i$ edges is the graph
formed by taking a complete graph on $a$ vertices, adding a single
vertex joined to $b$ of the vertices of the complete graph and adding
$n-a-1$ isolated vertices. A \emph{quasi-star} is the complement of a
quasi-complete graph.

Ahlswede and Katona~\cite{MR505076} showed that the families of 2-sets
(graphs) with the most intersecting pairs (adjacent edges) are either
quasi-complete graphs or quasi-stars. Moreover, they showed that there
exists some non-negative integer $R$ (depending on $n$) such that for
$i<\frac12\binom n2-R$ and for $\frac12 \binom n2\le i\le \frac
12\binom n2+R$ the extremal family is a quasi-star, while for all
other values of $i$ the extremal family is a quasi-complete graph.
Wagner and Wang~\cite{MR2657026} extended this by finding the value of
$R$ explicitly and showing that it is non-zero for a proportion $\sqrt
2-1$ of numbers $n$. Combining Theorem~\ref{t:layer2} with these
results we see that the extremal families even for $N$ in this range
are surprisingly complicated: for many values of $n$ (i.e., those for
which $R\not=0$) the extremal families can switch between the two
classes three times just in this single layer.
\subsection*{Layout of Paper}
In the first section we define the notation we shall use and recall
the definitions and some of the properties of the compressions that we
use.  In the second section we prove a slightly weaker version of
Theorem~\ref{t:unique} that is sufficient (in combination with
Theorem~\ref{t:layer2}) to prove Theorem~\ref{t:not-nested}.  In the
third section we prove Theorem~\ref{t:layer2}.
In the fourth section we prove the remaining cases of
Theorem~\ref{t:unique} and give the constructions showing that the
lower bound on $N$ in Theorem~\ref{t:unique} is tight.
We conclude the paper with a discussion of some open problems.
\section{Notation and Preliminaries}

Most of the notation we use is standard. We write $[n]$ for the set
$\{1,2,\ldots,n\}$ and $[m,n]$ for the set $\{m,m+1,\ldots, n\}$. For
any $r$ we use $[n]^{(r)}$ to denote the set of subsets of $[n]$ of
size $r$, and $[n]^{(\ge r)}$ to denote the set of subsets of $[n]$ of
size at least $r$.

For any family $\cA$ we let $\cI^{(s)}(\cA)$ denote the collection of
intersecting subfamilies of $\cA$ of size $s$. For clarity, when $s$
is clear from the context we  suppress the superscript.

In much of this paper we shall be aiming to change or \emph{compress} a
family $\cA$ into a nice form without decreasing the number of
intersecting subfamilies of a given size. 

We  use three type of compression. The first is very simple:
we replace a set $A\in \cA$ by a set $A'\supset A$ with
$A'\not\in\cA$. Obviously this preserves the size of $\cA$ and does
not decrease the number of intersecting subfamilies. We call this an
\emph{up-set-compression}.

The second operation we use is a very standard compression called an
\emph{$ij$-compression.} We take each set $A$ in $\cA$ and, if $i\not \in A$
and $j\in A$, we replace $A$ by the set $A\cup\{i\}\setminus\{j\}$
\emph{provided} that this set is not already in $\cA$.  We note that
these compressions do not change the size of any set in $\cA$.

Again it is easy to see that this preserves the size of $\cA$. This
time, it is not obvious that the compression does not decrease the
number of intersecting subfamilies. It is, however, proved
in~\cite{PARkat}.

We will generally be applying these compressions when $i<j$ and we
call such a compression a \emph{left-compression}. 

The final operation we use is the $(U,v,f)$-compression recently
introduced in~\cite{PARkat}. Suppose that $U\subset[n]$ has even size,
that $f\colon U\to U$ is a permutation of order 2 with no fixed point,
and that $v\in[n]\setminus U$. We move each set $A\in\cA$ with $v\not
\in A$ to $(A\setminus U\cup \{v\})\cup f(A\cap U)$ unless this set is
already in $\cA$. Informally, we add $v$ and swap the points inside
$U$. Again it is clear that this does not change the size of $\cA$.
Note also that every set moved by this compression contains $v$ after
the move.

We shall use the following key property of these
$(U,v,f)$-compressions (proved in~\cite{PARkat}). For any such
compression $C$ there exists an injection $\widehat C$ from
$\cI^{(s)}(\cA)$ to $\cI^{(s)}(C(\cA))$ and so, in particular, the
number of intersecting subfamilies of any given order does not
decrease.  The only property of $\widehat C$ that we shall use is that
$\widehat C(\cB)\in \cI(C(\cA))$ is formed from $\cB\in\cI(\cA)$ by
sending each set $A\in\cB$ to either $A$ or $C(A)$. We remark that
constructing the injection $\widehat C$ is non-trivial.

\section{Proof of Theorem~\ref{t:unique}}

In this section we prove a slightly weaker version of
Theorem~\ref{t:unique} covering all the cases where
$N\ge\sum_{k=\lceil n/2\rceil -1 }^n\binom nk$: that is, the $N$ for
which our putative extremal family would contain all of the first
layer below the middle.  This is sufficient for our main result
(Theorem~\ref{t:not-nested}). For completeness, we prove the remaining
cases in Section~\ref{s:middle}.

Define $r=r(N,n)$ to be the unique number $r$ satisfying
$\sum_{k=r+1}^n\binom nk \le N< \sum_{k=r}^n\binom nk$. Thus the
bound for $N$ above corresponds to $r< n/2-1$.

We start by showing that if $\cA$ has a particularly nice form then
there is a $(U,v,f)$-compression that strictly increases the number of
intersecting subfamilies of size $s$.
\begin{lemma}\label{l:strict}
Let $2\le s \le 2^{n-1}$ and $\ell<\frac{n}2-1$. Suppose that $\cA$
satisfies $[n]^{(>\ell+1 )}\subset \cA$, $[n-\ell,n]\not \in \cA$ and
$[\ell]\in \cA$.  Then there is a $(U,v,f)$-compression $C$
such that $|\cI^{(s)}(C(\cA))|>|\cI^{(s)}(\cA)|$.
\end{lemma}
\begin{proof}
Choose $C$ to be any $(U,v,f)$-compression 
with $v=n$ that moves  $[\ell]$ to $[n-\ell,n]$, and let
$\cC=C(\cA)$ be the resulting family.  We construct a family in
$\cI^{(s)}(\cC)$ that is not the image of any family in
$\cI^{(s)}(\cA)$ under the injection~$\widehat C$.  

Consider the family
\[
\cD=\{[n-\ell,n]\}\cup [n]^{(\ge n-\ell-1)}\setminus\{[n-\ell-1]\}.
\]
This is intersecting, since $\ell< n/2-1$ and so $n-\ell-1>n/2$. Thus
it extends to a maximal intersecting family $\cD'$ of size $2^{n-1}$
in $\cP([n])$. Since $\cD'$ contains all sets of size at least
$n-\ell-1$ except $[n-\ell-1]$ and is intersecting, $\cD'$ contains no
set of size less than or equal to $\ell+1$ except $[n-\ell,n]$. By
hypothesis, $\cA$ contains all of $[n]^{(>\ell+1)}$ and thus so does
$\cC$. Moreover, $[\ell]\in \cA$ is moved to $[n-\ell,n]$ so
$[n-\ell,n]\in \cC$. Hence $\cD'\subset \cC$. Also, since
$\big|\,[\ell+1,n-1]\,\big|=n-\ell-1$, we have $[\ell+1,n-1]\in \cD'$.
Let $\cD''$ be any subfamily of size $s$ of $\cD'$ containing
both $[n-\ell,n]$ and $[\ell+1,n-1]$. Note $\cD''\in\cI(\cC)$.

Suppose that there is an intersecting subfamily $\cB$ of $\cA$ with
$\widehat C(\cB)=\cD''$. Recall that $\widehat C(\cB)$ is formed from
$\cB$ by sending each set $A\in\cB$ to either $A$ or $C(A)$.
Now $[n-\ell,n]\in\cD''$ but $[n-\ell,n]\not\in \cA$ so $[n-\ell,n]\not\in\cB$.
Hence $[n-\ell,n]$ must have come from $[\ell]\in\cB$. Also,
$[\ell+1,n-1]\in\cD''$ and, since $n\not\in[\ell+1,n-1]$, this set has a
unique pre-image under $C$, namely the set $[\ell+1,n-1]$
itself. Therefore $[\ell+1,n-1]\in\cB$.  But we also have $[\ell]\in\cB$,
contradicting the fact that $\cB$ is intersecting.

We conclude that $\cD''$ is not the image under $\widehat C$ of any
family in $\cI^{(s)}(\cA)$. Hence $|\cI^{(s)}(\cC)|>|\cI^{(s)}(\cA)|$.
\end{proof}
In Section~\ref{s:middle} we slightly strengthen this result, proving that with
some extra conditions it holds for $\ell=\lceil n/2\rceil -1$.

\begin{corollary}\label{c:unique-all-top}
  Let $n,N\in \N$ with $r=r(N,n)<n/2$ and $2\le s\le 2^{n-1}$. Suppose
  that $[n]^{(\ge r+1)}\subset \cA$ and that $\cA$ contains a set of
  size strictly less than $r$. Then there is a family of size $N$ that
  contains strictly more intersecting subfamilies of size $s$ than
  does $\cA$.
\end{corollary}
\begin{proof}
By the definition of $r$ we see that $\cA$ does not contain all of
$[n]^{(r)}$. Hence by applying left-compressions we can ensure that
$[n-r+1,n]\not \in \cA$. Also, since $\cA$ contains some set of size
at most $r-1$, by applying up-set-compressions and left-compressions we
can ensure that $[r-1]\in \cA$; it is easy to check that we can do
this without putting $[n-r+1,n]$ into $\cA$. Thus Lemma~\ref{l:strict}
applies with $\ell=r-1$.
\end{proof}

\begin{lemma}\label{l:unique-not-all-top}
Suppose that $n,N\in \N$ with $r=r(N,n)<n/2-1$, that $[n]^{(\ge
  r+1)}\not\subset \cA$, and that $2\le s\le 2^{n-1}$. Then there is a
family that contains strictly more intersecting subfamilies of size
$s$ than does $\cA$.
\end{lemma}
\begin{proof}
  We aim to compress the family until it contains nearly all of
  $[n]^{(\ge r+1)}$. We then apply one more compression and use
  Lemma~\ref{l:strict} to show strict inequality for this final
  compression. We need to be careful that the earlier compressions do
  not `accidentally' put all sets in $[n]^{(\ge r+1)}$ into our family
    since then we would not necessarily obtain strict inequality when
    applying the final compression.

    We construct a sequence of families $\cA=\cA_0,\cA_1,\cA_2,\ldots,
    \cA_k$, with $[n]^{(\ge r+1)}\not\subset \cA_i$ for any $i$, by
    applying at each stage any `allowed' up-set-compression,
    left-compression or $(U,v,f)$-compression---that is, one which
    does not result in our family containing the whole of $[n]^{(\ge
      r+1)}$. We finish with a family $\cA_k$ that is unchanged by any
    compression $C$ with $[n]^{\ge(r+1)}\not\subset C(\cA_k)$. Note
    that $\cA_k$ is left-compressed since $ij$-compressions do not
    change the size of any set. Also, by considering
    up-set-compressions, we see that $\cA^+=\cA_k\cap [n]^{(\ge r+1)}$
    is an up-set and $\cA^-=\cA_k\cap [n]^{(\le r)}$ is an up-set when
    viewed as a subset of $[n]^{(\le r)}$. Obviously
    $\cA^+\not=[n]^{(\ge r+1)}$ and so $\cA^-$ is non-empty.

We claim that $\cA^+=[n]^{(\ge r+1)}\setminus \{[n-r,n]\}$. If only
one set from $[n]^{(\ge r+1)}$ is missing from $\cA^+$ then it must be
$[n-r,n]$.  Thus we may assume for a contradiction that at least two
of the sets in $[n]^{(\ge r+1)}$ are missing from $\cA^+$.  Since
$\cA^+$ is a left-compressed up-set we see that these missing sets
must include $[n-r,n]$ and $\{n-r-1\}\cup[n-r+1,n]$.  Similarly, as
$A^-$ is a non-empty left-compressed `up-set', we see that
$[r]\in\cA^-$.  We have now shown that $[r]\in\cA_k$, that
$[n-r,n]\not\in\cA_k$ and that $\{n-r-1\}\cup[n-r+1,n]\not\in \cA_k$.
The upper bound on $r$ implies that the sets $[r]$ and $[n-r,n]$ are
disjoint. Hence we can map the former to the latter using a
$(U,v,f)$-compression with $v=n-r$. This does not add the set
$\{n-r-1\}\cup[n-r+1,n]$ since all sets added by such a compression
contain $v=n-r$. Hence this is an allowed $(U,v,f)$-compression which
contradicts the definition of~$\cA_k$.

So $\cA_k$ contains all of $[n]^{(\ge r+1)}$ except for the set
$[n-r,n]$ and, as before, it must contain $[r]$.  Hence
Lemma~\ref{l:strict} applies with $\ell=r$.
\end{proof}
This essentially completes the proof of Theorem~\ref{t:unique} for
$r<\frac n2-1$. Indeed, by Lemma~\ref{l:unique-not-all-top}, $[n]^{(\ge
  r+1)}\subset \cA$ and so, by Corollary~\ref{c:unique-all-top}, $\cA$
has the required form.  

The only remaining cases are $\frac{n}2-1\le r\le \frac n2$. We deal
with these cases in Section~\ref{s:middle}.

\section{Proof of Theorem~\ref{t:layer2}}

  Fix $s$ and, as usual, let $\cI=\cI^{(s)}$ denote the collection of
  intersecting subfamilies of $\cA$ of size $s$. For $\cB\subset
  \cA\cap[n]^{(2)}$ let
\[
\cI_\cB=\{\cE\in \cI:\cE\cap [n]^{(2)}=\cB\}.
\]
We see that $\cI$ is the disjoint union of the sets $\cI_\cB$ over all
collections $\cB$ of $2$-sets in $\cA$. Moreover, $\cI_\cB$ is empty
unless $\cB$ is intersecting. Since all sets in $\cB$ have size two,
the structure of these intersecting families is simple. Indeed, for
all $0\le r\le n-1$ except $r=3$, there is a unique (up to reordering
the coordinates) intersecting family of size $r$ in $[n]^{(2)}$,
namely the star $\cS_r$ consisting of the sets $\{1t\}$ for $2\le t\le
r+1$. Trivially, for $n\ge 4$ there is no intersecting family of size
greater than $n-1$. For $r=3$ there are two intersecting families
(again up to reordering the coordinates), namely $\cS_3=\{12,13,14\}$
and $\cT=\{12,13,23\}$ which we shall call the star and the triangle
respectively.

Since, by hypothesis, we know that $\cA$ contains all sets of size at
least 3 and no sets of size 1, we have 
\[
\cI_\cB=\left\{\cE \subset \cP\left([n]^{(\ge 2)}\right):
\cE\cap[n]^{(2)}=\cB, \text{$\cE$ intersecting}\right\}.
\]
Hence the cardinality of $\cI_\cB$ depends only on which of
$\cS_0,\cS_1,\ldots \cS_{n-1},\cT$ is isomorphic to $\cB$. Let
\[\cI_r=\left\{\cE \subset \cP\left([n]^{(\ge 2)}\right):
\cE\cap[n]^{(2)}=\cS_r, \text{$\cE$ intersecting}\right\}
\]
and
\[\cI_\cT=\left\{\cE \subset \cP\left([n]^{(\ge 2)}\right):
\cE\cap[n]^{(2)}=\cT, \text{$\cE$ intersecting}\right\}.
\]

Let $a_r$ be the number of intersecting subfamilies of
$\cA\cap[n]^{(2)}$ isomorphic to $\cS_r$ and $b$ be the number isomorphic to $\cT$.
Then
\[
|\cI|=\sum_{r=0}^{n-1}a_r|\cI_r|+b|\cI_\cT|.
\]
Obviously $a_0=1$ and $a_1=|\cA\cap[n]^{(2)}|=i$
so the first two terms of the sum are independent of the collection
$\cA$. Trivially we have $a_r\le n\binom {n-1}r$ and $b\le \binom
n3$. If we compare $\cI_r$ and $\cI_{r-1}$ we see that a family in
$\cI_r$ has two extra restrictions: it must contain the set
$\{1,r+1\}$ (which gives us one fewer set to place) and each set must
intersect $\{1,r+1\}$ (which places an extra restriction on where
these other sets can lie). Thus we might expect $\cI_{r-1}$ to be much
larger than $\cI_{r}$.  That is precisely what Lemma~\ref{l:stars} will show.
First we need the following simple result.
\begin{lemma}\label{l:triangle}
  $|\cI_\cT|\le |\cI_3|$.
\end{lemma}
\begin{proof}
 There is a unique maximal intersecting family
  containing $\cT$ and thus every intersecting family containing
  $\cT$ is contained in this unique maximal family. Hence the number
  of intersecting families of size $s$ containing $\cT$ is the
  smallest it can possibly be, namely $\binom{2^{n-1}-3}{s-3}$.
\end{proof}

We give one definition that will be useful in the proof.
\begin{defn}
  Suppose that $\cE,\cF$ are families in $\cP([n])$. We say they are
  \emph{cross-intersecting} (or that they \emph{cross-intersect}) if
  for every $E\in\cE$ and $F\in\cF$ we have $E\cap F\not=\emptyset$.
\end{defn}

\begin{lemma}\label{l:stars}
If  $r> 3$ then  $|\cI_{r-1}|\ge \left(2^{n-r-2}-\frac{n-r}{2}\right)|\cI_r|$
and if $r=3$ then $|\cI_{r-1}|\ge \left(2^{n-5}-\frac{n-1}{2}\right)|\cI_r|$.
\end{lemma}
\begin{proof}
For $r>3$ we construct a mapping $\Phi\colon[r+2,n]^{(\ge 2)}\times
\cI_r\to \cI_{r-1}$ under which every family in $\cI_{r-1}$ has at
most two pre-images. In the case $r=3$ we instead construct
$\Phi\colon[r+2,n]^{(\ge 2)}\times \cI_r\to \cI_{r-1}\cup
\cI_\cT$. Recalling that $|\cI_\cT|\le |\cI_3|$ this suffices to prove
the lemma.

Throughout the proof we write $X^c$ to denote the complement of the
set $X$ relative to $[r+2,n]$: that is, for  $X\subset[r+2,n]$ we write
$X^c=[r+2,n]\setminus X$.

Suppose that $\cE\in \cI_r$ and $U\in [r+2,n]^{(\ge 2)}$. Let
$U'=\{1,r+1\}\cup U$. First we tweak $\cE$ slightly to make sure that
it contains $U'$. If $U'\in\cE$ let $\bar\cE=\cE$; otherwise let
$\bar\cE=\cE\setminus\left\{\{1,r+1\}\right\}\cup \{U'\}$. Note that
the new family $\bar\cE$ is still intersecting since $\{1,r+1\}\subset
U'$.

We split $\bar\cE$ into pieces as follows: 
\begin{align*}
\cE_0&=\{E\in\bar\cE:1 \in E, E\cap[2,r]\not=\emptyset\}\\
\cE_1&=\{E\in\bar\cE:1 \in E, E\cap[2,r+1]=\emptyset\}\\
\cE_2&=\{E\in\bar\cE:1 \not\in E\}\\
\cE_3&=\{E\in\bar\cE:1,r+1 \in E, E\cap[2,r]=\emptyset\}.
\end{align*}
Clearly $\cE_2=\{E\in\cE:1 \not\in E\}$.
As $\cE\in \cI_r$ we know that $\cE$ is intersecting and 
$\{1,j\}\in \cE$ for $2\le j\le r+1$. Hence
$\cE_2=\{E\in\bar\cE:E\cap[r+1]=[2,r+1]\}$.

We define $\cE_1',\cE_2',\cE_3'$ to be the restrictions of
$\cE_1,\cE_2,\cE_3$ to $[r+2,n]$: that is, $\cE_i'=\{E\cap
[r+2,n]:E\in \cE_i\}$ for $i=1,2,3$. Since the intersection of a  set in $\cE_i$
($i=1,2,3$) with  $[1,r+1]$ depends only on $i$, we
see that the sets $\cE_1,\cE_2,\cE_3$ are determined by
$\cE_1',\cE_2',\cE_3'$ respectively.

We make a couple of remarks about this partition that will be helpful
later in the proof. First,  $U'\in \cE_3$ and so
$U\in \cE_3'$. Secondly, $\cE_1'$ and $\cE_2'$ are cross-intersecting.

To define our new intersecting family $\cF=\Phi(U,\cE)$ we split
into two cases according to  whether $U^c\cap E\not=\emptyset$ for all
$E\in \cE_1'$.  Note that if this condition does not hold then $U$
meets every element of $\cE_2'$. Indeed, suppose $F\in \cE_2'$ with
$F\cap U=\emptyset$ and $E\in \cE_1'$ with $E\cap U^c=\emptyset$. Then
$F\subset U^c$ and $E\subset U$ so $E\cap F=\emptyset$ which
contradicts the cross-intersection property observed above.

\bigskip\noindent%
\textbf{Case 1:} $U^c\cap E\not=\emptyset$ for all $E\in \cE_1'$.  
\medskip

\noindent%
This is the simpler case: here starting from $\bar\cE$ we replace each
set $X\in\bar\cE$ satisfying $\{1,r+1\}\subset X\subset U'$ by its
complement. Formally, let $\cF_i=\cE_i$ for $i=0,1,2$, let 
\[
\cF_3= \{\{1,r+1\}\cup E:E\in\cE_3',E\cap U^c\not=\emptyset\}  
\]
and let
\[
\cF_4= \{[2,r]\cup E^c:E\in\cE_3',E\subseteq U\}.
\]
Set $\cF=\cF_0\cup \cF_1\cup \cF_2\cup \cF_3\cup \cF_4$.

The families $\cF_0,\cF_1,\ldots,\cF_4$ are pairwise disjoint and
there is an obvious bijection from $\cE_3$ to $\cF_3\cup\cF_4$. Hence
$|\cF|=|\cE|$.

Moreover, this new family $\cF$ is intersecting: obviously
$\cF_0\cup\cF_1\cup\cF_2\cup\cF_3\subset \cE$ and so is intersecting,
and $\cF_4$ is an intersecting family, so we only need to check that
$\cF_4$ cross-intersects each of the other
$\cF_i$. Trivially $\cF_4$ cross-intersects $\cF_0$ and $\cF_2$. We
see that $\cF_3$ and $\cF_4$ cross-intersect as $U^c$ intersects every
set in $\cF_3$ and is contained in every set in $\cF_4$.  Finally,
$\cF_1$ and $\cF_4$ cross-intersect because we are assuming  that $U^c$
intersects everything in $\cE_1'$.

Next we show that if $r\ge 4$ then $\cF\cap[n]^{(\le2)}$ is
$\cS_{r-1}$, and if $r=3$ then $\cF\cap[n]^{(\le2)}$ is either $\cS_{r-1}$
or $\cT$. Since $\cE\cap[n]^{(\le2)}=\cS_r$, we see that
$\bar\cE\cap[n]^{(\le2)}$ is either $\cS_r$ or $\cS_{r-1}$. Clearly
$S_{r-1}\subset \cE_0$. If $\{1,r+1\}\in\bar\cE$ then it is in $\cE_3$
but not $\cF_3$. Hence
$(\cF_0\cup\cF_1\cup\cF_2\cup\cF_3)\cap[n]^{(\le2)}=S_{r-1}$. Finally,
if $r\ge 4$ then $\cF_4\cap[n]^{(\le2)}=\emptyset$; if $r=3$ then
$\cF_4\cap[n]^{(\le2)}$ is either empty or  $\{\{2,3\}\}$.
Thus we have shown that if $r\ge4$ then $\cF\in\cI_{r-1}$ and if $r=3$
then $\cF\in\cI_{r-1}\cup\cI_\cT$.

Finally, \emph{if we know that $\cF$ comes from this case} then we can
reconstruct $\bar\cE$. Indeed, given $\cF$, set $\cF_4=\{F\in \cF:
F\cap[1,r+1]=[2,r]\}$. Then form $\bar\cE$ from $\cF$ by replacing
each set in $\cF_4$ by its complement. We also know $U$ since
$[2,r]\cup U^c$ is the unique minimal element of $\cF_4$. Once we know
$\bar \cE$ and~$U$, it easy to reconstruct $\cE$.

\bigskip\noindent%
\textbf{Case 2:} $U\cap E\not=\emptyset$ for all $E\in \cE_2'$. 
\medskip

\noindent%
This time the construction is a little more complicated. We define
\begin{align*}
\cF_0&=\cE_0\\
\cF_1&= \{\{1,r+1\}\cup E:E\in\cE_1'\}  \\
\cF_2&= \{[2,r]\cup E:E\in\cE_2'\}  \\
\cF_3&= \{\{1\}\cup E:E\in\cE_3',U\subseteq E\}  \\
\cF_4&= \{[2,r+1]\cup E^c:E\in\cE_3',E^c\cap U\not=\emptyset\}
\end{align*}
and, as before, set $\cF=\cF_0\cup \cF_1\cup \cF_2\cup \cF_3\cup \cF_4$.

Again $\cF_0=\cE_0$ but this time $\cF_1\not=\cE_1$ and
$\cF_2\not=\cE_2$. However, there are bijections between $\cF_1$ and
$\cE_1$, and between $\cF_2$ and $\cE_2$. As before there is a
bijection between $\cF_3\cup\cF_4$ and $\cE_3$. Hence $|\cF|=|\cE|$.

Again $\cF$ is intersecting. Indeed, each $\cF_i$ is
trivially intersecting, $\cF_0$ is cross-intersecting with each of the
others, and each of the pairs $(\cF_1,\cF_3)$, $(\cF_1,\cF_4)$ and
$(\cF_2,\cF_4)$ is trivially cross-intersecting.  We see that $\cF_3$
and $\cF_4$ cross-intersect as $U$ is contained in every set in
$\cF_3$ and intersects every set in $\cF_4$. Also, $\cF_1$ and $\cF_2$
cross-intersect because $\cE_1'$ and $\cE_2'$ are
cross-intersecting. Finally, $\cF_2$ and $\cF_3$ cross-intersect
because we are assuming that $U$ intersects everything in $\cE_2'$.

Next we show that if $r\ge 4$ then $\cF\cap[n]^{(\le2)}$ is
$\cS_{r-1}$, and if $r=3$ then $\cF\cap[n]^{(\le2)}$ is either $\cS_{r-1}$
or $\cT$. We consider each $\cF_i\cap[n]^{(\le2)}$ in turn. We have
$\cF_0\cap[n]^{(\le2)}=\cE_0\cap[n]^{(\le2)}=\cS_{r-1}$. As $\{1\}\not
\in \cE$, $\emptyset\not\in\cE_1'$ so
$\cF_1\cap[n]^{(\le2)}=\emptyset$.  If $r\ge 4$ then
$\cF_2\cap[n]^{(\le2)}=\emptyset$; if $r=3$ then
$\cF_2\cap[n]^{(\le2)}$ is either empty or $\{\{2,3\}\}$. As $|U|\ge
2$, $\cF_3\cap[n]^{(\le2)}=\emptyset$. Finally, it is obvious that
$\cF_4\cap[n]^{(\le2)}=\emptyset$.
Again we have shown that if $r\ge4$ then $\cF\in\cI_{r-1}$ and if $r=3$
then $\cF\in\cI_{r-1}\cup\cI_\cT$.

Now, given $\cF$ we can determine $\cF_0,\cF_1,\ldots,\cF_4$ by
considering intersections with $[1,r+1]$. Thus, if we knew we were in
this case, we could reconstruct $\cE$ and $U$. (This time $\{1\}\cup
U$ is the unique minimal element of $\cF_3$).

\medskip
However, we cannot (necessarily) tell from which case the family $\cF$ came.
Thus the function is not necessarily injective but each family
in $\cI_{r-1}$ has at most two pre-images as required.
\end{proof}

The rest of the proof is straightforward calculation.  Recall that, by
hypothesis, $n\ge 21$. It follows from Lemma~\ref{l:stars} that
$|\cI_3|\le 2^{6-n}|\cI_2|$.  It also follows that $|\cI_4|\le
2^{7-n}|\cI_3|$, and that $|\cI_{r}|\le|\cI_{r-1}|$ for $r\ge5$. Thus,
for all $r\ge 4$, $|\cI_r|\le 2^{13-2n}|\cI_2|$.  Furthermore, by
Lemmas~\ref{l:triangle} and~\ref{l:stars}, $|\cI_\cT|\le
2^{6-n}|\cI_2|$. Recall $a_0=1$ and $a_1=i$. Now
\begin{align*}
|\cI|&=\sum_{r=0}^{n-1}a_r|\cI_r|+b|\cI_\cT|\\
&=|\cI_0|+i|\cI_1|+a_2|\cI_2|+\sum_{r=3}^{n-1}a_r|\cI_r|+b|\cI_\cT|\\
&=|\cI_0|+i|\cI_1|+|\cI_2|\left(a_2+
\sum_{r=3}^{n-1}a_r\frac{|\cI_r|}{|\cI_2|}
+b\frac{|\cI_\cT|}{|\cI_2|}\right)
\end{align*}
and 
\[
\sum_{r=3}^{n-1}a_r\frac{|\cI_r|}{|\cI_2|}
+b\frac{|\cI_\cT|}{|\cI_2|}\le n\binom{n-1}{3}2^{6-n} +\binom n3
2^{6-n}+\sum_{r=4}^{n-1}n\binom {n-1}r 2^{13-2n}.
\]
It is easy to verify that the quantity on the right-hand-side of the
inequality is less than $1$ for all $n\ge 21$.

This essentially completes the proof. Indeed, suppose $\cB'\subset
[n]^{(2)}$ is a family of size $i$ with strictly more intersecting
pairs than $\cB$. Let $a_r'$ and $b'$ be the corresponding values for
$\cB'$. Then $a_0'=1$, $a_1'=i$, and $a_2'\ge a_2+1$. Let $\cI'$
be the collection of intersecting families of size $s$ in $[n]^{(\ge
  3)}\cup \cB'$. Using the above we have
\begin{align*}
|\cI'|&= \sum_{r=0}^{n-1}a_r'|\cI_r|+b'|\cI_\cT|\\
&\ge  |\cI_0|+i|\cI_1|+a_2'|\cI_2|\\
&\ge |\cI_0|+i|\cI_1|+(a_2+1)|\cI_2|\\
& > \sum_{r=0}^{n-1}a_r|\cI_r|+b|\cI_\cT|\\
&=|\cI|,
\end{align*}
contradicting the maximality of $\cA$.\qed
\section{The middle-layer cases of Theorem~\ref{t:unique}}\label{s:middle}
In this section we conclude the proof of Theorem~\ref{t:unique} by
showing that it holds for $\frac n2-1\le r\le\frac n2$.
We consider separately the cases of $n$ even and $n$ odd.

First we deal with some cases where $N$ is not too close to the
bound stated in Theorem~\ref{t:unique}. In each case we prove a slight
variant on Lemma~\ref{l:strict} and use it to deduce the result.
\subsection*{Case 1:\textnormal{ $n=2t$ and $\sum_{k=t}^n\binom nk\le N< \sum_{k=t-1}^n\binom nk$}}
The bounds on $N$ imply that $r=t-1=\frac n2-1$. Also, the lower bound
on $N$ is equal to $2^{n-1}+\frac12\binom nt$ so this covers nearly
all of the remaining possibilities for $N$.
\begin{lemma}\label{l:strict2}
Let $n$ be even, $2\le s \le 2^{n-1}$ and $\ell=\frac n2-1$. Suppose that
$\cA$ satisfies $[n]^{(\ge \ell+1 )}\setminus\{[n-\ell,n]\} \subset \cA$,
$[n-\ell,n]\not \in \cA$ and $[\ell]\in \cA$. Then there is a
$(U,v,f)$-compression $C$ such that
$|\cI^{(s)}(C(\cA))|>|\cI^{(s)}(\cA)|$.
\end{lemma}
\begin{proof}
Exactly as before choose $C$ to be any $(U,v,f)$-compression with
$v=n$ that moves $[\ell]$ to $[n-\ell,n]$, and let $\cC=C(\cA)$ be the
resulting family.  We construct a family in $\cI^{(s)}(\cC)$ that is
not the image of any family in $\cI^{(s)}(\cA)$ under the
injection~$\widehat C$.

Consider the family
\[
\cD=\{[n-\ell,n],[n-\ell-1,n-1]\}\cup[n]^{(> \ell+1)}
\]
Note that $\big|[n-\ell,n]\big|=\big|[n-\ell-1,n-1]\big|=\ell+1=\frac n2$ and
so $\cD$ is intersecting. Thus it extends to a maximal intersecting
family $\cD'$ of size $2^{n-1}$ in $\cP([n])$. Since $\cD'$ contains
all sets of size at least $\ell+2=n/2+1$ and is intersecting, it
contains no set of size less than or equal to $\ell$. By hypothesis $\cA$
contains all of $[n]^{(\ge \ell+1)}$ except $[n-\ell,n]$ and thus so does
$\cC$. Moreover $[\ell]\in \cA$ is moved to $[n-\ell,n]$ so $[n-\ell,n]\in
\cC$. Thus, in fact, $\cC$ contains all of $[n]^{(\ge \ell+1)}$. Hence
$\cD'\subset \cC$.

Let $\cD''$ be any subfamily of $\cD'$ of size $s$ containing
$[n-\ell,n]$ and $[n-\ell-1,n-1]$. Note $\cD''\in \cI(\cC)$. Exactly as
before we see that any intersecting family mapping to $\cD''$ under
$\widehat C$ would have to contain $[\ell]$ and $[n-\ell-1,n-1]$ which is
not possible as $[n-\ell-1,n-1]=[\ell+1,n-1]$ is disjoint from~$[\ell]$.
\end{proof}
\begin{proof}[Proof of Theorem~\ref{t:unique} in this case]
Using Lemma~\ref{l:strict2} instead of Lemma~\ref{l:strict} we can
prove a result analogous to
Lemma~\ref{l:unique-not-all-top}. Combining this with
Corollary~\ref{c:unique-all-top} is sufficient to complete the proof
in this case.
\end{proof}
\subsection*{Case 2:  \textnormal{$n=2t+1$ and $\sum_{k=t+1}^n\binom nk+\binom{n-1}{t-1}\le N< \sum_{k=t}^n\binom nk$}}
In this case the bounds on $N$ imply that $r=t=\frac{n-1}2$. Also, the
lower bound on $N$ is equal to $2^{n-1}+\binom{n-1}{t-1}$.

\begin{lemma}\label{l:strict3}
Let $n$ be odd, $2\le s \le 2^{n-1}$ and $\ell=\frac{n-1}2$. Suppose that
$\cA$ satisfies $[n]^{(\ge \ell+1 )}\setminus\{[n-\ell,n]\} \subset \cA$ and
$[n-\ell,n]\not \in \cA$, and that there exist $A,A'\in [n]^{(\ell)}\cap\cA$
with $n\not \in A$, $n\not\in A'$ and $A\cap A'=\emptyset$.  Then
there is a $(U,v,f)$-compression $C$
with  $|\cI^{(s)}(C(\cA))|>|\cI^{(s)}(\cA)|.$
\end{lemma}
\begin{proof}
As $A$ and $A'$ are distinct we may assume without loss of generality
that $A'\not = [\ell-1]$. In particular this implies that
$A'\cap[n-\ell,n]\not=\emptyset$.  

This time we choose $C$ to be any $(U,v,f)$-compression with $v=n$
that moves $A$ to $[n-\ell,n]$. Let $\cC=C(\cA)$ be the resulting family.
We again construct a family in $\cI^{(s)}(\cC)$ that is not the image
of any family in $\cI^{(s)}(\cA)$ under the injection~$\widehat C$.

Let
\[
\cD=[n]^{(\ge \ell+1)}\setminus \{[n]\setminus A'\}\cup \{A'\}.
\]
This is itself a maximal intersecting family. Since $C(A)=[n-\ell,n]$ we
see that $\cC$ contains all of $[n]^{(\ge \ell+1)}$. Further, since $A$ is the
only set in $\cA$ of size $\ell$ that moves, we see that $A'\in \cC$.
Hence $\cD\subset\cC$.

Let $\cD''$ be any subfamily of size $s$ of $\cD$ containing
$[n-\ell,n]$ and $A'$. Then $\cD''\in\cI(\cC)$. Similarly to the previous
cases, we see that any intersecting family mapping to $\cD''$ under
$\widehat C$ would have to contain $A$ and $A'$ which is not possible.
\end{proof}
\begin{proof}[Proof of Theorem~\ref{t:unique} for this case]
Here we require a little more care.  If $[n]^{(\ge r+1)}\subset \cA$
then Corollary~\ref{c:unique-all-top} tells us that $\cA$ has the
required form.  Hence we may assume that $\cA$ does not contain all of
$[n]^{(\ge r+1)}$.  We aim to compress $\cA$ into a form where we can
apply Lemma~\ref{l:strict3}.

Exactly as in the proof of Lemma~\ref{l:unique-not-all-top} we
construct a sequence of families $\cA=\cA_0,\cA_1,\cA_2,\ldots, \cA_k$
such that $[n]^{(\ge r+1)}\not\subset \cA_i$ for any $i$, by applying
at each stage an allowed compression. Let $\cA'=\cA_k$ be the final
compressed family. As before the only set of $[n]^{(\ge r+1)}$ not in
$\cA'$ is $[n-r ,n]$.

Since $t=r$ the lower bound on $N$ implies that $|\cA'\cap [n]^{(\le
  r)}|\ge \binom{n-1}{r-1}+1$. We claim that, in fact, there are at
least this many sets in the layer $[n]^{(r)}$.

If $\cA'$ contains no set of size less than $r$ then the claim holds
trivially. Otherwise, $\cA'\cap [n]^{(\le r-1)}$ is a non-empty
left-compressed up-set in $[n]^{(\le r-1)}$, and so
$[r-1]\in\cA'$. Moreover, there is a $(U,v,f)$-compression with $v=r$
moving $[r-1]$ to $\{r\}\cup[n-r+2,n]$.  Since all sets added by such
a compression contain $v=r$ this does not add $[n-r,n]$ and, thus,
$\cA'$ is closed under this compression. Hence
$\{r\}\cup[n-r+2,n]\in\cA'$.

Since $\cA'$ is left-compressed and $\{r\}\cup[n-r+2,n]\in\cA'$, the
family $\cA'$ also contains every set that can be obtained from
$\{r\}\cup[n-r+2,n]$ by (repeated) left-compression. This includes
every set in $[n]^{(r)}$ containing an element less than or equal to
$r$. It is easy to see that there are strictly more than
$\binom{n-1}{r-1}$ such sets and so our claim holds.

Now since $\cA'$ contains strictly more than $\binom{n-1}{r-1}$ sets
in $[n]^{(r)}$, the Erd\H{o}s-Ko-Rado Theorem tells us that $\cA'$
must contain two disjoint sets $A$ and $A'$.  We now complete the proof by
applying Lemma~\ref{l:strict3} unless either $A$ or $A'$ contains $n$.

So assume, without loss of generality, that $n\in A$.  Now there must
be some $m\not=n$ contained in neither $A$ nor $A'$.  Since $\cA'$ is
left compressed, and so in particular $mn$-compressed, the set
$A''=A\cup\{m\}\setminus\{n\}$ is also in $\cA'$. The set $A''$ is
disjoint from $A'$ and does not contain $n$. Hence we can apply
Lemma~\ref{l:strict3} with sets $A''$ and $A'$ to conclude the proof
in this case.
\end{proof}

\subsection*{The final few cases}
For completeness we prove the final few cases of
Theorem~\ref{t:unique}. The arguments in this section are completely
different from those given earlier: they are not compression-based.
The only remaining cases are 
\begin{itemize}
\item  $n=2t$ with $2^{n-1}+\frac12\binom
{n}{t}-t<N< 2^{n-1}+\frac12\binom {n}{t}$; and
\item $n=2t+1$ with
$2^{n-1}+\binom {n-1}{t-1}-t-1<N< 2^{n-1}+\binom {n-1}{t-1}$.
\end{itemize}

Any family of size $N$ must contain at least $N-2^{n-1}$ complementary
pairs.  For $N$ in the range we are now considering, Katona, Katona
and Katona~\cite{3kat} give an example of a family containing exactly
this many complementary pairs and no other non-intersecting
pairs. Moreover, they observe that any such family contains the
maximal number of intersecting subfamilies of size $s$ for every $s$.
Conversely, it is easy to check that if a family $\cA$ contains the
maximal number of subfamilies of size $s$ for any fixed $s$ then all
non-intersecting pairs in $\cA$ must be complementary.

It follows that each set in a complementary pair must be a minimal
element of $\cA$. Furthermore the subfamily $\cB$ given by
\[
\cB=\begin{cases}\cA\cap[n]^{(\le t)}\qquad&\text{$n=2t+1$ odd}\\
\cA\cap\left([n]^{(< t)}\cup \{A\in [n]^{(t)}:1\in A\}\right)\qquad&\text{$n=2t$ even}\\
\end{cases}
\]
must be intersecting since in each case we take the intersection of
$\cA$ with a family not containing any complementary pairs.  Note that
$|\cA|\le 2^{n-1}+|\cB|$ and hence, using the lower  bound on $N$, that
$|\cB|>\binom{n-1}{t-1}  -(n-t)$.
 Since every set in a complementary pair is
minimal and every complementary pair contains an element of $\cB$, the
following lemma completes the proof. 

\begin{lemma}
  Let $\cB\subset [n]^{(\le t)}$ be an intersecting family with
  $\cB\not\subset [n]^{(t)}$. Then $\cB$ has at most
  $\binom{n-1}{t-1}-(n-t)$ minimal elements.
\end{lemma}
\begin{proof}
Let \[
\cU=\{B\in\cB: \text{$B$ minimal, $|B|<t$}\}
\]
and 
\[
\cV=\{B\in\cB: \text{$B$ minimal, $|B|=t$}\}.
\]
Let $\partial \cU=\{A\in [n]^{(t)}:B\subset A $ for some $B\in\cU\}$
be the upper shadow of $\cU$ in layer $t$. Note that, by
definition of $\cU$ and $\cV$, the families $\partial\cU$ and $\cV$
are disjoint. It is easy to see that $\partial \cU\cup\cV$ is an
intersecting family and so, by the Erd\H{o}s-Ko-Rado Theorem, must
have size at most $\binom{n-1}{t-1}$. Thus the total number of
minimal elements of $\cB$ is
\[
|\cV|+|\cU|= |\cV|+|\partial \cU|+|\cU|-|\partial
\cU|\le \binom{n-1}{t-1}-(|\partial \cU|-|\cU|)
\]
and so it suffices to show that $|\partial \cU|-|\cU|\ge n-t$. As
$\cU$ is an antichain, it follows easily from the Kruskal-Katona
Theorem that $\cU$ is no larger than its upper shadow in the $(t-1)$th
layer. Hence, we may assume $\cU\subset [n]^{(t-1)}$.

Form a bipartite graph $G$ with vertex-sets $[n]^{(t-1)}$ and
$[n]^{(t)}$. For $A\in[n]^{(t-1)}$ and $B\in[n]^{(t)}$ we take $AB$ to
be an edge of $G$ whenever $A\subset B$. Every vertex in $[n]^{(t-1)}$
has degree $n-t+1$ and every vertex in $[n]^{(t)}$ has degree $t$. Fix
$A\in\cU$ and let $\Gamma(A)$ denote the set of neighbours of $A$ in
$G$. Form a graph $G'$ by deleting $\{A\}\cup\Gamma(A)$ from $G$. Since every
vertex in $[n]^{(t-1)}\setminus \{A\}$ is joined to at most one
deleted vertex, we see that the degree in $G'$ of every vertex in
$[n]^{(t-1)}\setminus \{A\}$ is at least $n-t$. Hence, since $n-t\ge
t$, a standard application of Hall's Theorem shows that there exists a
matching in $G'$ from $[n]^{(t-1)}\setminus\{A\}$ to
$[n]^{(t)}\setminus \Gamma(A)$. Thus
\[
|\partial \cU|\ge |\cU|-1+|\Gamma(A)|=|\cU|+n-t.\qedhere
\]
\end{proof}
\subsection*{Constructions showing the bound is tight}
Recall that a family $\cA$ of size $N$ containing precisely
$N-2^{n-1}$ complementary pairs and no other non-intersecting pairs
maximises the number of intersecting subfamilies of every possible
size. We remark that an equivalent condition is that $\cA$ meets every
complementary pair in $\cP([n])$ and the only non-intersecting pairs
in $\cA$ are complementary.

\renewcommand{\labelenumi}{(\roman{enumi})}

\begin{theorem}\label{t:construct}
Suppose that
\begin{enumerate}
\item $n=2t$ and $2^{n-1}<N\le 2^{n-1}+\tfrac12\tbinom{n}{t}-t$; or
\item $n=2t+1$ and $2^{n-1}<N\le 2^{n-1}+\tbinom {n-1}{t-1}-t-1$. 
\end{enumerate}
Then there exists a family $\cA$ containing the maximal number of
intersecting families of every possible size that is not of the form
$[n]^{(\ge t+1)}\cup\cB$ for any $\cB\subset [n]^{(t)}$.
\end{theorem}
\begin{proof}
(i) Consider the family
\[
\cA=\{[t-1]\}\cup [n]^{(\ge t)}\setminus\{A\in[n]^{(t)}:A\cap [t-1]=\emptyset\}.
\]
This family has size $2^{n-1}+\frac12\binom nt -t$ and meets every
complementary pair in $\cP([n])$, and the only non-intersecting pairs
in $\cA$ are complementary pairs. Thus, $\cA$ satisfies the conclusion
of the theorem for $N=2^{n-1}+\tfrac12\tbinom{n}{t}-t$.  Moreover, for
$2^{n-1}<N< 2^{n-1}+\tfrac12\tbinom{n}{t}-t$ we may obtain a suitable
family by deleting from $\cA$ the appropriate number of sets in
$[n]^{(t)}$ that contain the element 1.

(ii) In this case we apply an identical argument starting from the family
\[
\cA=\{[t-1]\}\cup \{A\in[n]^{(t)}:1\in A\}\cup
[n]^{(\ge t+1)} \setminus \{A\in[n]^{(t+1)}:A\cap [t-1]=\emptyset\}
\]
of size $2^{n-1}+\binom{n-1}{t-1}-t-1$ and deleting sets in
$[n]^{(t)}$ that contain the element~1 but do not contain all of
$[t-1]$.
\end{proof}

\section{Concluding Remarks and Open Questions}
As we remarked earlier there seems to be no reason to believe that the
maximising families for different values of $p$ are the same. However,
in all cases where the maximising families are known, including the
new examples in this paper, they are in fact the same for all
$p$. Even more is true: the known examples simultaneously maximise the
number of intersecting subfamilies of every possible size. We
therefore recall the following question first asked in~\cite{PARkat}.
\begin{question}
Suppose $N>2^{n-1}$. Does there exist a family $\cA$ of size $N$
which simultaneously maximises the  number of intersecting
subfamilies of size $s$ for every $s$?
\end{question}

One could, of course, ask an analogous question for families $\cA$
restricted to lie in a single layer of the cube $\cP([n])$. But here
it is not always possible to simultaneously maximise the number of
intersecting families of every size. Indeed we have seen that in
$[n]^{(2)}$ the family of size $\binom{n-1}{2}$ with the most
intersecting pairs is $\{A\in [n]^{(2)}:n\not \in A\}$. However, this
is obviously not the family containing the most intersecting
subfamilies of size $n$: it does not contain any at all.

The exact extremal families for most values of $N$ remain
unknown. Indeed, even the following question is open.
\begin{question}
  For every $N$ satisfying the lower bound of Theorem~\ref{t:unique},
  is there a unique (up to reordering the coordinates) family $\cA$
  that maximises the probability that $\cA_p$ is intersecting?
\end{question}

The family is not unique for $N$ less than the bound given: for
example, the optimal family constructed in Theorem~\ref{t:construct}
is not the same as that constructed
in~$\cite{3kat}$. Theorem~\ref{t:unique} shows that the family is
unique for certain values of $N$, primarily $N$ of the form
$N=\sum_{k=r}^n \binom nk$.

We remark that, for fixed $s$, the family containing the most
intersecting subfamilies of size $s$ is not always unique: indeed, for
$s=2$, $n\equiv 0,1$ mod 4 and $N=\sum_{k=3}^n\binom nk+\frac12\binom
n2$, the results of Ahlswede and Katona~\cite{MR505076} imply that the
families obtained by taking the union of $[n]^{(\ge3)}$ with a
quasi-clique or a quasi-star are both optimal. However, it is easy to
check that, at least for large $n$, the family with the quasi-star has
far more intersecting triples and, indeed, far more intersecting
families of size $s$ for all $s>2$. Hence the family maximising the
probability $\cA_p$ is intersecting is unique in this case.

 \bibliography{mybib}{}
\bibliographystyle{habbrv}

\end{document}